\documentclass[a4paper]{amsart}

\usepackage[latin1]{inputenc}  
\usepackage[T1]{fontenc}       

\usepackage{amsmath}
\usepackage{amsthm}
\usepackage{amssymb}
\usepackage{amsfonts}

\usepackage{dsfont}

\newtheorem{theorem}{Theorem}[section]
\newtheorem{lemma}[theorem]{Lemma}
\newtheorem{proposition}[theorem]{Proposition}

\newtheorem{remark}[theorem]{Remark}

\newtheorem{example}[theorem]{Example}

\providecommand{\norm}[1]{\lVert#1\rVert} 
\providecommand{\abs}[1]{\lvert#1\rvert}

\DeclareMathOperator{\Herm}{Herm}

\DeclareMathOperator{\Ran}{Ran} 
\DeclareMathOperator{\ran}{Ran}

\DeclareMathOperator{\Ker}{Ker}
\DeclareMathOperator{\Dom}{Dom}
\DeclareMathOperator{\dom}{dom}

\newcommand{\au}{\underline{a}}

\newcommand{\R}{\mathbb{R}}

\newcommand{\C}{\mathbb{C}}

\newcommand{\N}{\mathbb{N}}

\newcommand{\Me}{{\mathcal M}}
\newcommand{\He}{{\mathcal H}}
\newcommand{\Ke}{{\mathcal K}}
\newcommand{\Ge}{{\mathcal G}}
\newcommand{\Ie}{{\mathcal I}}
\newcommand{\Ee}{{\mathcal E}}
\newcommand{\De}{{\mathcal D}}

\newcommand{\We}{{\mathcal W}}

\author{Amru Hussein}
\address{A.~Hussein, FB 08 - Institut f\"{u}r Mathematik,
Johannes Gutenberg-Universit\"{a}t Mainz,Staudinger Weg 9, 55099 Mainz, Germany}
\email{hussein@mathematik.uni-mainz.de}

\title[Bounds on the negative eigenvalues of Laplacians on graphs]{Bounds on the negative eigenvalues of Laplacians on finite metric graphs}
\begin{document}

\begin{abstract}
For a self--adjoint Laplace operator on a finite, not necessarily compact, metric graph lower and upper bounds on each of the negative eigenvalues are derived. For compact finite metric graphs Poincar\'{e} type inequalities are given.
\end{abstract}
   \maketitle

A metric graph is a locally linear one dimensional space with singularities at the vertices. One can think roughly of a metric graph as a union of finitely many intervals $[0,a_i]$, $a_i>0$ or $[0,\infty)$ glued together at their end points. The subject of self-adjoint Laplacians on metric graphs has attracted a lot of attention in the last decades. Without going into details the author refers to \cite{QG,VKRS1999,PKQG1} and the references therein. 

In the article \cite{PKQG1} it has been proven that on a finite not necessarily compact, metric graph all self-adjoint Laplacians are semi-bounded from below and furthermore the negative spectrum is purely discrete. P.~Kuchment, see \cite{PKQG1}, and V.~Kostrykin together with R.~Schrader, see \cite{VKRS2006}, gave two different lower bounds on the spectrum of self--adjoint Laplacians on finite metric graphs. The exercise discussed here is to estimate each negative eigenvalue of a self--adjoint Laplacian from below and from above. The smallest eigenvalue of a self--adjoint Laplace operator $-\widetilde{\Delta}$ is exactly the growth bound of the semigroup generated by the operator $\widetilde{\Delta}$. The lower and upper bounds on the bottom of the negative spectrum give a priori estimates on the growth bound of this semigroup. This is the main application the author has in mind.

The problem of calculating the number of negative eigenvalues has been solved by A.~Luger and J.~Behrndt in \cite{AL2009}, where variational principles for self--adjoint operator pencils, see \cite{ML2004}, were used. In the present study two--sided estimates on the negative eigenvalues are derived applying the same variational methods to the eigenvalues themselves rather than to their number.  Three types of estimates are obtained, each is optimal in a certain situation. In particular the formerly known lower bounds on the spectrum given in \cite{PKQG1} and \cite{VKRS2006} are re--obtained. The content of this paper is part of the author's PhD thesis, see \cite[Chapter 2]{Ich}. 

The note is structured as follows. First the basic concepts of self-adjoint Laplacians on metric graphs are introduced, followed by the discussion of the negative eigenvalues. In Section 3 the bounds on the negative eigenvalues are stated. The eigenvalue zero of self--adjoint Laplacians on compact metric graphs is discussed because it is technically very close to the problem of the negative eigenvalues. Poicar\'{e} type inequalities on compact metric graphs are obtained in Section~\ref{secPoicare}. Eventually the variational methods are presented and used to prove the main results.

\subsection*{Acknowledgements}
I would like to thank Vadim Kostrykin for providing the interesting question and for pointing out the article \cite{AL2009} and I am grateful to Annemarie Luger for making this article accessible for me prior to publication.

\section{Laplace operators on finite metric graphs} 

The notation is borrowed from the article \cite{VKRS2006} and it is summarized here briefly. A graph is a $4$-tuple $\Ge = \left( V, \Ie,\Ee, \partial \right)$, where $V$ denotes the set of \textit{vertices}, $\Ie$ the set of \textit{internal edges} and $\Ee$ the set of \textit{external edges}, where the set $\Ee \cup \Ie$ is summed up in the notion \textit{edges}. The \textit{boundary map} $\partial$ assigns to each internal edge $i\in \Ie$ an ordered pair of vertices $\partial (i)= (v_1,v_2)\in V\times V$, where $v_1$ is called its \textit{initial vertex} and $v_2$ its \textit{terminal vertex}. Each external edge $e\in \Ee$ is mapped by $\partial$ onto a single, its initial, vertex. A graph is called finite if $\abs{V}+\abs{\Ie}+\abs{\Ee}<\infty$ and afinite graph is called \textit{compact} if $\Ee=\emptyset$ holds.

The graph is endowed with the following metric structure. Each internal edge $i\in \Ie$ is associated with an interval $[0,a_i]$ with $a_i>0$ such that its initial vertex corresponds to $0$ and its terminal vertex to $a_i$. Each external edge $e\in \Ee$ is associated to the half line $[0,\infty)$, such that $\partial(e)$ corresponds to $0$. The numbers $a_i$ are called \textit{lengths} of the internal edges $i\in \Ie$ and they are summed up into the vector \\  $\au=\{a_i\}_{i\in \Ie}\in \R_+^{\abs{\Ie}}$, and  one sets
\begin{eqnarray*}
a_{\min}:= \min_{i\in \Ie} a_i & \mbox{and} & a_{\max}:= \max_{i\in \Ie} a_i.
\end{eqnarray*}
The 2-tuple consisting of a finite graph endowed with a metric structure is called a \textit{metric graph} $(\Ge,\au)$. The metric on $(\Ge,\au)$ is defined via minimal path lengths. 

Given a metric graph $(\Ge,\au)$ one considers the Hilbert space
\begin{eqnarray*}
 \He \equiv \He(\Ee,\Ie,\au)= \He_{\Ee} \oplus \He_{\Ie}, & \displaystyle{\He_{\Ee}= \bigoplus_{e\in\Ee} \He_e,} & \He_{\Ie}= \bigoplus_{i\in\Ie} \He_i,
\end{eqnarray*}     
where $\He_j= L^2(I_j)$ with 
$$I_j= \begin{cases} [0,a_j], & \mbox{if} \ j\in \Ie, \\ [0,\infty), &\mbox{if} \ j\in \Ee.  \end{cases}$$

To provide suitable domains of definition for operators and forms one introduces appropriate Sobolev spaces. By $\We_j$ with $j\in \Ee \cup \Ie$ one denotes the set of all $\psi_j\in \He_j$ such that $\psi_j$ is absolutely continuous and $\psi_j^{\prime}$ is square integrable. On the whole metric graph one considers the direct sum 
\begin{eqnarray*}
\We= \bigoplus_{j\in \Ee \cup \Ie} \We_j.
\end{eqnarray*}
By $\De_j$ with $j\in \Ee \cup \Ie$ denote the set of all $\psi_j\in \He_j$ such that $\psi_j$ and its derivative $\psi_j^{\prime}$ are absolutely continuous and $\psi_j^{\prime\prime}$ is square integrable. Let $\De_j^0$ denote the set of all elements $\psi_j\in \De_j$ with
\begin{eqnarray*}
\psi_j(0)=0, &  \psi^{\prime}(0)=0, & \mbox{for}  \ j\in \Ee,    \\
\psi_j(0)=0, &  \psi^{\prime}(0)=0, & \psi_j(a_j)=0, \  \psi^{\prime}(a_j)=0, \ \mbox{for} \ j\in \Ie. 
\end{eqnarray*}
Let be $\Delta$ be the differential operator 
\begin{eqnarray*}
\left( \Delta \psi\right)_j (x) = \frac{d^2}{dx}\psi_j(x), & j\in \Ee\cup \Ie, & x\in I_j
\end{eqnarray*}
with domain $\De$ and $\Delta^0$ its restriction on the domain $\De^0$, where 
\begin{eqnarray*}
\De= \bigoplus_{j\in \Ee \cup \Ie} \De_j & \mbox{and} &  \De^0= \bigoplus_{j\in \Ee \cup \Ie} \De_j^0.
\end{eqnarray*}

The operator $\Delta^0$ is closed and symmetric with equal deficiency indices $(d,d)$, where $d=\abs{\Ee}+2\abs{\Ie}$. Its Hilbert space adjoint is $\Delta=(\Delta^0)^*$. 

Self-adjoint extensions can be discussed in terms of boundary conditions. For this purpose one defines the auxiliary Hilbert space
\begin{equation}\label{Ke}
\Ke \equiv \Ke(\Ee, \Ie) = \Ke_{\Ee}  \oplus \Ke_{\Ie}^- \oplus \Ke_{\Ie}^+
\end{equation}
with $\Ke_{\Ee} \cong \C^{\abs{\Ee}}$ and $\Ke_{\Ie}^{(\pm)} \cong \C^{\abs{\Ie}}$. For $\psi\in \De$ one defines the vectors of boundary values
\begin{eqnarray*}
\underline{\psi}= \begin{bmatrix} \{\psi_{e}(0)\}_{e\in\Ee} \\ 
\{\psi_{i}(0)\}_{i\in\Ie} \\
\{\psi_{i}(a_i)\}_{i\in\Ie}
\end{bmatrix} &\mbox{and} & 
\underline{\psi^{\prime}}= \begin{bmatrix} \{\psi_{e}^{\prime}(0)\}_{e\in\Ee} \\ 
\{\psi_{i}^{\prime}(0)\}_{i\in\Ie} \\
\{-\psi_{i}^{\prime}(a_i)\}_{i\in\Ie}
\end{bmatrix}.
\end{eqnarray*}
One sets 
\begin{equation*}
[\psi]:= \underline{\psi} \oplus \underline{\psi^{\prime}} \in \Ke \oplus \Ke
\end{equation*}
and denotes by the redoubled space $\Ke^2= \Ke \oplus \Ke$ the \textit{space of boundary values}. 

Let be $A$ and $B$ be linear maps in $\Ke$. By $(A, \, B)$ one denotes the linear map from $\Ke^2=\Ke \oplus \Ke$ to $\Ke$ defined by 
$$ (A, \, B) (\chi_1 \oplus \chi_2) = A\chi_1 + B \chi_2, $$
for $\chi_1,\chi_2\in\Ke$. One sets 
\begin{equation}\label{Me}
\Me(A,B):=\Ker (A, \, B).
\end{equation}
With any subspace $\Me\subset \Ke^2$ of the form \eqref{Me} one can associate an extension $\Delta(\Me)$ of $\Delta^0$ which acts as $\Delta$ on the domain
\begin{equation*}
\mbox{Dom} (\Delta(\Me))= \{ \psi \in \De \mid [\psi] \in \Me   \}.
\end{equation*}
For $\Me=\Me(A,B)$ an equivalent description is that $\mbox{Dom} (\Delta(\Me))$ consists of all functions $\psi\in \De$ that satisfy the boundary conditions
\begin{equation}
A \underline{\psi} + B \underline{\psi^{\prime}} = 0, 
\end{equation}
and one also writes $\Delta(\Me)=\Delta(A,B)$. All self-adjoint extensions of $\Delta^0$ can be parametrized by matrices $A,B\in \mbox{End}(\Ke)$, with
\begin{enumerate}
\item $(A, \,  B)\colon \Ke^2 \to \Ke$ is surjective 
\item $AB^*$ self-adjoint,
\end{enumerate}  
see \cite{VKRS1999} and the references therein. The parametrization by the matrices $A$ and $B$ is not unique, since two parametrisations $A,B$ and $A^{\prime},B^{\prime}$ describe the same operator if and only if the Lagrangian subspaces $\Me(A,B)$ and $\Me(A^{\prime},B^{\prime})$ agree. In \cite[Corollary 5]{PKQG1} a unique way to parametrize self-adjoint Laplacians in terms of an orthogonal projection $P$ acting in $\Ke$ and of a self-adjoint operator $L$ acting in the subspace $\Ker P$ is given. For any self-adjoint Laplacian one has $-\Delta(A,B)=-\Delta(A^{\prime},B^{\prime})$ with $A^{\prime}= L+P$ and $B^{\prime}=P^{\perp}$, where the change over is given by
\begin{align}\label{KuchmentL2}
L= \left(B\mid_{\ran B^*}\right)^{-1}A P^{\perp}
\end{align}
and $P$ denotes the orthogonal projector onto $\Ker B\subset \Ke$ and $P^{\perp}= \mathds{1}-P$ is the complementary projector. The strictly positive part of the operator $L$ is denoted by $L_+$, its strictly negative part by $L_-$ and the orthogonal projector onto the kernel of $L$ as a map in the space $\ran P^{\perp}\subset \Ke$ is denoted by $L_0$.  

\section{Negative eigenvalues}
The negative spectrum of a self--adjoint Laplace operator $-\Delta(A,B)$ consists of only finitely many eigenvalues of finite multiplicity, because the minimal operator $-\Delta^0$ has only finite deficiency indices. A fundamental system of the equation 
\begin{eqnarray*}
\left( -\frac{d^2}{dx^2} + \kappa^2  \right)u(\cdot,\kappa)=0, & & \kappa\neq 0
\end{eqnarray*}
is given by the functions $e^{-\kappa x}$ and $ e^{\kappa x}$. For $\kappa>0$ only the first mentioned function $e^{-\kappa x}$ is square integrable on the half line $[0,\infty)$ and hence on the external edges. Consequently an \textit{Ansatz} for a square integrable eigenfunction to a negative eigenvalue $-\kappa^2$ is 
\begin{align*}
\psi(x,i\kappa)= \begin{cases} s_j(i\kappa) e^{-\kappa x}, & j\in \Ee, \\
\alpha_j(i\kappa) e^{-\kappa x} + \beta_j(i\kappa) e^{\kappa x}, & j\in \Ie.  \end{cases}
\end{align*}
The function $\psi(\cdot,i\kappa)$ has the traces
\begin{align*}
\underline{\psi(\cdot,i\kappa)}= X\left( i\kappa,\underline{a} \right) \begin{bmatrix} \{s_j(i\kappa)\}_{j\in \Ee} \\ \{\alpha_j(i\kappa)\}_{j\in \Ie} \\ \{\beta_j(i\kappa)\}_{j\in \Ie} \end{bmatrix}, && 
\underline{\psi(\cdot,i\kappa)}^{\prime}= -\kappa \cdot Y\left(i\kappa,\underline{a}\right) \begin{bmatrix} \{s_j(i\kappa)\}_{j\in \Ee} \\ \{\alpha_j(i\kappa)\}_{j\in \Ie} \\ \{\beta_j(i\kappa)\}_{j\in \Ie} \end{bmatrix},
\end{align*}
where 
\begin{eqnarray*}\index{$X\left( i\kappa,\underline{a} \right)$}\index{$Y\left( i\kappa,\underline{a} \right)$}
X\left( i\kappa,\underline{a} \right)=\begin{bmatrix} \mathds{1} & 0 & 0 \\ 0 & \mathds{1} & \mathds{1} \\ 0 & e^{-\kappa\underline{a}} & e^{\kappa\underline{a}}\end{bmatrix} & \mbox{and}& 
Y\left( i\kappa,\underline{a} \right)=\begin{bmatrix} \mathds{1} & 0 & 0 \\ 0 & \mathds{1} & -\mathds{1} \\ 0 & -e^{-\kappa\underline{a}} & e^{\kappa\underline{a}}\end{bmatrix}
\end{eqnarray*}
are given with respect to the decomposition $\Ke=\Ke_{\Ee}  \oplus \Ke_{\Ie}^- \oplus \Ke_{\Ie}^+$ given in \eqref{Ke}. Here $e^{\pm\kappa\underline{a}}$ denote $\abs{\Ie}\times \abs{\Ie}$--diagonal matrices with entries $\{e^{\pm\kappa\underline{a}}\}_{i,j}=\delta_{i,j} e^{\pm \kappa a_i }$. The function $\psi(\cdot,i\kappa)$ is indeed an eigenfunction to the eigenvalue $-\kappa^2 <0$ if and only if one has $\psi(\cdot,i\kappa)\in\Dom(-\Delta(A,B))$, where $\kappa>0$ is the positive square root of $\kappa^2$. This is the case if and only if the \textit{Ansatz} function $\psi(\cdot,i\kappa)$ satisfies the boundary conditions, which are encoded into the equation
\begin{align}\label{eq:detZLB}
 Z\left(i\kappa,\underline{a},A,B\right) \begin{bmatrix} \{s_j(i\kappa)\}_{j\in \Ee} \\ \{\alpha_j(i\kappa)\}_{j\in \Ie} \\ \{\beta_j(i\kappa)\}_{j\in \Ie} \end{bmatrix}=0,
\end{align}\index{$Z\left(i\kappa,\underline{a},A,B\right)$}
where 
\begin{align*}
 Z\left(i\kappa,\underline{a},A,B\right)=AX(i\kappa,\underline{a})-\kappa\,BY(i\kappa,\underline{a}).
\end{align*}
The matrices $X\left(i\kappa,\underline{a}\right)$ and $Y\left(i\kappa,\underline{a}\right)$ are invertible for $\kappa> 0$. Hence equation~\eqref{eq:detZLB} has non--trivial solutions if and only if
\begin{eqnarray}\label{eq:detMLB}
\det(A + BM(\kappa,\underline{a}))=0 &\mbox{with} & M(\kappa,\underline{a})= -\kappa \cdot Y(i\kappa,\underline{a})X(i\kappa,\underline{a})^{-1}.
\end{eqnarray}\index{$M(\kappa,\underline{a})$}
The operator $M(\kappa,\underline{a})$ has the diagonalisation 
$$M\left( \kappa, \underline{a} \right)= Q D\left(\kappa,\underline{a}\right)Q$$
with 
\begin{eqnarray*}
Q=\begin{bmatrix}
\mathds{1} &0 & 0 \\
0 & \frac{\mathds{1}}{\sqrt{2}} & \frac{\mathds{1}}{\sqrt{2}} \\
0 & \frac{\mathds{1}}{\sqrt{2}} & -\frac{\mathds{1}}{\sqrt{2}}
\end{bmatrix} & \mbox{and} &
D(\kappa,\au)=\begin{bmatrix}
-\kappa &0 & 0 \\
0 & \lambda\left(\kappa,\underline{a}\right) & 0\\
0 & 0 & \mu\left(\kappa,\underline{a}\right)
\end{bmatrix},
\end{eqnarray*}\index{$D\left(\kappa,\underline{a}\right)$}
where
\begin{eqnarray*}
\lambda\left(\kappa,\underline{a}\right)= -\kappa \tanh(\kappa\underline{a}/2) &\mbox{and}  &
\mu\left(\kappa,\underline{a}\right)= \frac{-\kappa}{\tanh(\kappa\underline{a}/2)}
\end{eqnarray*}
are $\abs{\Ie}\times \abs{\Ie}$--diagonal matrices with entries 
\begin{eqnarray*}
\{\lambda\left(\kappa,\underline{a}\right)\}_{i,j}=  -\delta_{i,j}\kappa \tanh(\kappa a_j/2) & \mbox{and} & \{\mu\left(\kappa,\underline{a}\right)\}_{i,j}=  \delta_{i,j}\frac{-\kappa}{\tanh(\kappa a_j/2)},
\end{eqnarray*}
respectively. Note that the operator $Q$ is a symmetry, since $Q^2=\mathds{1}$ and $Q^{\ast}=Q$ hold. One sets 
\begin{eqnarray}\label{M0LB}
\displaystyle{M(0,\au):=\lim_{\kappa \to 0} M(\kappa,\au)}, &\mbox{hence} & M(0,\au)= \begin{bmatrix} 0 & 0 & 0 \\ 0 & -\frac{1}{\au} & \frac{1}{\au}  \\ 0 & \frac{1}{\au} & -\frac{1}{\au} \end{bmatrix},
\end{eqnarray}\index{$M(0,\au)$}
since
\begin{eqnarray*}
\displaystyle{\lim_{\kappa\to 0} -\kappa \tanh(\kappa a_i/2)=0} &\mbox{and} & \displaystyle{\lim_{\kappa\to 0}\frac{-\kappa}{\tanh(\kappa\ a_i/2)}= -\frac{2}{a_i}}.
\end{eqnarray*}
Consider instead of the parametrization by $A,B$ the equivalent boundary conditions which are defined by $A^{\prime}, B^{\prime}$, where $A^{\prime}=L+P$ and $B^{\prime}=P^{\perp}$ with $P$ the orthogonal projector onto $\Ker B$ and $L$ the Hermitian operator in $\Ran P^{\perp}$ which is given by \eqref{KuchmentL2}. The decomposition $$\Ke=(\ran P) \oplus (\ran P^{\perp})$$ induces in \eqref{eq:detMLB} the block structure 
\begin{align*}
\det \left(\begin{bmatrix}  P^{\perp} LP^{\perp} & 0 \\ 0 & P  \end{bmatrix}
+ \begin{bmatrix}  P^{\perp} M(\kappa,\au)P^{\perp} & P^{\perp} M(\kappa,\au)P \\ 0 & 0 \end{bmatrix} \right)=0.
\end{align*}
Consequently $-\kappa^2$ is a negative eigenvalue of $-\Delta(A,B)$ if and only if the Hermitian operator 
\begin{eqnarray*}\index{$L(\kappa,\au)$}
L(\kappa,\au):=\left(L + P^{\perp} M(\kappa,\au)P^{\perp} \right), & \kappa>0,
\end{eqnarray*}
considered as an operator in the space $\Ran P^{\perp}$ is not invertible. The multiplicity of $-\kappa^2$ equals the dimension of  $\Ker L(\kappa,\au)$. One sets $L(0,\au)=L+ P^{\perp}M(0,\au)P^{\perp}$. On star graphs, that is for $\Ie=\emptyset$, one has simply $L(\kappa)=L-\kappa P^{\perp}$. 

Note that by the use of $L(\cdot,\au)$ instead of $Z(\cdot,A,B,\au)$ the problem of solving the eigenvalue equation transforms to a non-linear eigenvalue problem with values in the Hermitian operators. Variational methods apply to the function $L(\cdot,\au)$. This has been used to determine the number of negative eigenvalues in \cite[Theorem 1]{AL2009}, which is reformulated here as 
\begin{proposition}[{\cite[Theorem 1]{AL2009}}]\label{Luger}
The number of negative eigenvalues of the self-adjoint Laplace operator $-\Delta(A,B)$ (counted with multiplicities) is given by the number of positive eigenvalues of $L(0,\au)$ (counted with multiplicities). In particular, $-\Delta(A,B)$ is non--negative if and only if the operator $L(0,\au)$ is non--positive.
\end{proposition}
\begin{remark}
Let $-\Delta(A,B)$ be self--adjoint, $P$ be the orthogonal projector onto $\Ker B$ and let $L$ be given by \eqref{KuchmentL2}. Then the operator $L(\kappa,\au)=\left(L + P^{\perp} M(\kappa,\au)P^{\perp} \right)$ as an operator in the space $\Ran P^{\perp}$ is invertible if and only if the operator $$\tau_{A,B}(\kappa,\au)=AB^* + B M(\kappa,\au)B^*$$ considered as an operator in the space $\Ran B$ is invertible. In the original formulation of \cite[Theorem 1]{AL2009} the operator $\tau_{A,B}(0,\au)$ is used instead of $L(0,\au)$. 
\end{remark}

\section{Bounds on the negative eigenvalues}

Applying the mentioned variational methods used in \cite{AL2009} delivers three different lower and upper bounds on each of the negative eigenvalues of a self-adjoint Laplacian $-\Delta(A,B)$. Each type of bound reflects the decay properties of one of the three blocks of the block-diagonal operator $D(\kappa,\au)$. The estimates are obtained by proposing matrix valued functions, which are majorants or minorants to the function $L(\cdot,\au)$ in the sense of quadratic forms. Denote by 
$$-\kappa_1^2\leq \ldots \leq -\kappa_n^2<0$$ 
the negative eigenvalues of $-\Delta(A,B)$ (counted with multiplicities). For technical reasons it is convenient to consider the positive square roots of the absolute value of the negative eigenvalues. These are
$$\kappa_1\geq \ldots \geq \kappa_n>0$$ 
(counted with multiplicities). By Proposition~\ref{Luger} the number of negative eigenvalues of the operator $-\Delta(A,B)$ and the number of positive eigenvalues of $L(0,\au)$ coincide. The positive eigenvalues of $L(0,\au)$ are
$$l_1 \geq \ldots \geq l_n>0$$ 
(counted with multiplicities) and $l_1=\norm{\left(L(0,\au)\right)_+}$. For $l>0$ the equation
\begin{eqnarray*}
l = \frac{\kappa}{\tanh(\kappa a/2)} -\frac{2}{a}
\end{eqnarray*}
has a unique non--trivial solution, which is denoted by $\eta(l,a)$. \index{$\eta(l,a)$} 

\begin{theorem}\label{theorem1LB}
Let $-\Delta(A,B)$ be self--adjoint. Then the square roots of the absolute values of the negative eigenvalues of $-\Delta(A,B)$ obey the two--sided estimates 
\begin{eqnarray*}
l_i \leq \kappa_i \leq \eta(l_i,a_{\min}), & 1 \leq i \leq n.
\end{eqnarray*}
The lower bound on the spectrum 
$$-\eta(l_1,a_{\min})^2\leq -\Delta(A,B)$$
is optimal if and only if 
\begin{eqnarray*}
\mbox{span} \left\{ \begin{bmatrix} 0 \\ e_i \\ -e_i \end{bmatrix} \Bigg\vert a_i=a_{\min} \right\} \cap \Ker(L(0,\au)-l_1) \neq \{0\},
\end{eqnarray*} 
where $e_i$ denotes the $i$--th canonical basis vector in $\Ke_{\Ie}^+$ and $\Ke_{\Ie}^-$, respectively, and $l_1=\norm{\left(L(0,\au)\right)_+}$ is the largest positive eigenvalue of $L(0,\au)$.
\end{theorem}
The proof of the above theorem is based on certain estimates on the operator valued function $L(\cdot,\au)$. Other types of estimates on $L(\cdot,\au)$ deliver similar theorems. The proofs are given in the next but one section along with a discussion of the variational methods used. 

Denote by 
$$m_1 \geq \ldots \geq m_n>0$$ 
the $n$ largest positive eigenvalues of $L$ (counted with multiplicities) and note that $m_1=\norm{L_+}$. For $m>0$ the equation 
\begin{eqnarray*}
m = \kappa \tanh(\kappa a/2)
\end{eqnarray*}
has a unique non--trivial solution $\nu(m,a)$.\index{$\nu(m,a)$} Furthermore for $m>\tfrac{2}{a}$ the equation 
\begin{eqnarray*}
m = \frac{\kappa}{\tanh(\kappa a/2)}
\end{eqnarray*}
has the unique non--trivial solution $\eta(m-\tfrac{2}{a},a)$. 

\begin{theorem}\label{theorem2LB}
Let $-\Delta(A,B)$ be self--adjoint. Then the square roots of the absolute values of the negative eigenvalues of $-\Delta(A,B)$ obey the estimate 
\begin{eqnarray*}
0<\kappa_i \leq \nu(m_i,a_{\min}), & 1 \leq i \leq n.
\end{eqnarray*}
The lower bound on the spectrum 
$$-\nu(m_1,a_{\min})^2 \leq -\Delta(A,B)$$ 
is optimal if and only if 
\begin{eqnarray*}
\mbox{span} \left\{ \begin{bmatrix} 0 \\ e_i \\ e_i \end{bmatrix} \Bigg\vert a_i=a_{\min} \right\} \cap \Ker(L-m_1) \neq \{0\},
\end{eqnarray*} 
where again $e_i$ denotes the $i$--th canonical basis vector in $\Ke_{\Ie}^+$ and $\Ke_{\Ie}^-$, respectively, and $m_1=\norm{L_+}$ is the largest positive eigenvalue of $L$.
\\
Whenever $m_i-\tfrac{2}{a_{\min}}>0$, the number $\kappa_i$ satisfies the lower bound
\begin{eqnarray*}
\eta\left(m_i-\tfrac{2}{a_{\min}},a_{\min}\right) \leq \kappa_i.
\end{eqnarray*}
\end{theorem}

\begin{remark}\label{Bspbolte}
For the smallest negative eigenvalue one re--obtains from Theorem~\ref{theorem2LB} the lower bound given in \cite[Theorem 3.10]{VKRS2006}, which in the above notation is 
$$-\nu(m_1,a_{\min})^2 \leq -\Delta(A,B).$$  
The optimality of this bound has been shown in \cite[Remark 4.1]{Bolte2009}, by means of the example $L=\mathds{1}$ and $P=0$ on a compact graph. An explicit computation of the negative eigenvalues of $-\Delta(L,\mathds{1})$ exhibits the optimality for this example. Applying Theorem~\ref{theorem2LB} shows an easier way. It is sufficient to notice that for any $i\in \Ie$
\begin{eqnarray*}
v_i=\begin{bmatrix} e_i \\ e_i\end{bmatrix} \in \Ker(L-1),
\end{eqnarray*}
where $1$ is the largest positive eigenvalue of $L=\mathds{1}$. Considering $-\Delta(\mathds{1},\mathds{1})$ on the interval $[0,a]$ one reads from
\begin{eqnarray*}
L(k,\au)=  Q  \begin{bmatrix} 1-\kappa \tanh(\kappa a /2) & 0 \\ 0 & 1-\frac{\kappa}{\tanh(\kappa a /2)} \end{bmatrix}Q
\end{eqnarray*}
that if $1-\tfrac{2}{a}>0$, one has a second negative eigenvalue $-\eta\left(1-\tfrac{2}{a},a\right)^2$. Theorem~\ref{theorem1LB} gives the lower bound $-\eta\left(1,a\right)^2<-\nu(1,a)^2 $ and therefore less accurate information on the bottom of the spectrum.
\end{remark}

An example for the optimality of the lower bound given in Theorem~\ref{theorem1LB} is 
\begin{example}\label{Bspn2}
Consider on the interval $[0,a]$ the operator $-\Delta(L_c,\mathds{1})$ with the non--local boundary conditions that are defined by 
\begin{eqnarray*}
L_c= \frac{c}{2} \cdot \begin{bmatrix} +1 & -1 \\ -1 & +1 \end{bmatrix},
\end{eqnarray*}
where one has 
\begin{eqnarray*}
L_c=  Q\begin{bmatrix} 0 & 0 \\ 0 & c \end{bmatrix}Q &\mbox{and hence } & L_c(\kappa,a)=Q  \begin{bmatrix} -\kappa \tanh(\kappa a/2) & 0 \\ 0 & c-\frac{\kappa}{\tanh(\kappa a/2)} \end{bmatrix}Q.
\end{eqnarray*}
As long as $c - \tfrac{2}{a}>0$ holds, there exists exactly one negative eigenvalue of $-\Delta(L_c,\mathds{1})$, which is the solution of 
$$\frac{\kappa}{\tanh(\kappa a/2)}=c.$$ The lower bound given in Theorem~\ref{theorem1LB} is optimal, whereas the lower bound of Theorem~\ref{theorem2LB} predicts a smaller lower bound below zero, even in the case when the operator $-\Delta(L_c,\mathds{1})$ is non--negative. Therefore Theorem~\ref{theorem2LB} provides less information in this case.
\end{example}

\begin{remark}
The solutions of the transcendent equations given in Theorem~\ref{theorem1LB} and Theorem~\ref{theorem2LB} are in general only implicitly given. However they can be majorized and minorized by piecewise algebraic functions. Note that 
\begin{eqnarray*}
\tanh(y) \geq \begin{cases} \frac{1}{4} y, &  \mbox{if} \ 0\leq y\leq 1, \\ \frac{1}{4} , & \mbox{if} \  y\geq 1 \end{cases}
\end{eqnarray*}
which yields
\begin{eqnarray*}
- \kappa\tanh\left(\frac{a}{2}\kappa\right) \leq s(\kappa,a):= -\begin{cases} \frac{1}{4} \frac{a}{2}\kappa^2, &  0 \leq \kappa\leq \frac{2}{a}, \\ \frac{1}{4} \kappa, &  \kappa\geq \frac{2}{a} \end{cases}
\end{eqnarray*}
and hence $m_1-\tanh(\kappa a_{\min}/2) \leq m_1 + s(\kappa,a_{\min})$. The unique solution of the equation $m + s(\kappa,a)=0$ for $m>0$ is 
\begin{eqnarray*}
\xi(m,a)= \begin{cases} \sqrt{\frac{8 m}{a}}, &  m\leq \frac{1}{2a}, \\ 4m, &  m\geq \frac{1}{2a}. \end{cases}
\end{eqnarray*}
Taking advantage of the monotony of the functions involved one concludes that the unique solution $\xi(m_1,a_{\min})$ of $m_1 + s(\kappa,a_{\min})=0$ is larger than the solution $\nu(m_1,a_{\min})$ of the equation $m_1-\tanh(\kappa a_{\min}/2)=0$. Thus $-\xi(m_1,a_{\min})^2\leq -\Delta(A,B)$ holds and hence also 
\begin{align*}
 -\Delta(A,B) \geq -\xi(m_1,a_{\min})^2 \geq 
\begin{cases}
-4\frac{m_1}{a_{\min}}(\abs{\Ee}+\abs{\Ie}), & m_1\leq \frac{1}{2 a_{\min}}, \\
-8 m_1^2(\abs{\Ee}+\abs{\Ie}) , & m_1\geq \frac{1}{2 a_{\min}}
\end{cases}
\end{align*} 
holds for $\abs{\Ee}+\abs{\Ie}\geq 2$. This is the lower bound on the spectrum given in \cite[Corollary 10]{PKQG1}. It has been proven there using quadratic forms associated with self--adjoint Laplace operators. The proof given here exhibits that this lower bound can be re--obtained from the bound given in \cite[Theorem 3.10]{VKRS2006}. More precisely the bound in \cite[Corollary 10]{PKQG1} results from a reduction of the bound given in \cite[Theorem 3.10]{VKRS2006}.   
\end{remark}

The bounds given in Theorem~\ref{theorem1LB} and in Theorem~\ref{theorem2LB} can be coarsened in order to obtain estimates in terms of affine linear functions. Consider the linear operator 
\begin{eqnarray*}
R(\kappa,\au)= L+ P^{\perp} M_1(\au) P^{\perp}- \kappa P^{\perp}, & \mbox{where } M_1(\au)=\begin{bmatrix} 0 & 0 & 0 \\ 0 & \frac{1}{\au} & \frac{1}{\au}  \\ 0 & \frac{1}{\au} & \frac{1}{\au} \end{bmatrix}
\end{eqnarray*}
as an operator in the space $\ran P^{\perp}$. Since $M(0,\au)\leq 0$ and $0\leq M_1(\au)$ in the sense of quadratic forms one has  
$$ L(0,\au) \leq  L \leq R(0,\au). $$ 
Denote by $r_1 \geq \ldots \geq r_n>0$ the $n$ largest positive eigenvalues of $R(0,\au)$ (counted with multiplicities).  
\begin{theorem}\label{theorem3LB}
Let $-\Delta(A,B)$ be self--adjoint. Then the square roots of the absolute values of the negative eigenvalues of $-\Delta(A,B)$ obey the two--sided estimates
\begin{eqnarray*}
l_i \leq \kappa_i \leq r_i, & \mbox{for }  1\leq i\leq n.
\end{eqnarray*}
\end{theorem}
Both estimates from below and from above are optimal for star graphs, that is for graphs with $\Ie=\emptyset$, since then $R(\kappa,\au)=L(\kappa,\au)=L-\kappa P^{\perp}$ holds. The estimates given in Theorem~\ref{theorem3LB} are compared to the ones given in Theorems~\ref{theorem1LB} and \ref{theorem2LB} easy to compute.

When $a_i \to \infty $, uniformly for all $i\in \Ie$ the lower and upper bounds obtained in \ref{theorem3LB} converge from below and from above to the positive eigenvalues of $L$. This follows from 
\begin{eqnarray*}
\lim_{a_{\min}\to\infty} M(0,\au) =0 & \mbox{and} & \lim_{a_{\min}\to\infty} M_1(\au) =0
\end{eqnarray*}
Hence the bounds are improving for large internal edge lengths. In the limit the negative eigenvalues behave like on the disconnected graph on which each internal edge has been replaced by two external edges.

\begin{example}
Consider the graph that consists of two external edges $\Ee=\{1,2\}$ connected by an internal edge $I=\{3\}$ of length $\au=\{a\}$. This means one has two vertices $\partial (1)=\partial_-(3)$ and $\partial (2)=\partial_+(3)$ each of degree two. On each vertex one imposes $\delta^{\prime}$--interactions with coupling parameter $\gamma\neq 0$, compare \cite[Section 3.2.1]{PKQG1}. These are locally given by the boundary conditions that are defined by 
\begin{eqnarray*}
A_{\nu}=
\begin{bmatrix} 
0 & 0 \\ 1 & 1
\end{bmatrix}
 & \mbox{and}& B_{\nu}= 
\begin{bmatrix} 
1 & -1 \\  -\gamma & 0 
\end{bmatrix}.
\end{eqnarray*}
This gives 
\begin{eqnarray*}
L= \displaystyle{-\frac{1}{\gamma}}
\begin{bmatrix} 
1 & 1 & 0 & 0 \\
1 & 1 & 0 & 0 \\
0 & 0 & 1 & 1 \\
0 & 0 & 1 & 1 
\end{bmatrix} & \mbox{and} & P=0,
\end{eqnarray*}
and 
\begin{eqnarray*}
M(0,\au)= \displaystyle{\frac{1}{a}}
\begin{bmatrix} 
0 & 0 & 0 & 0 \\
0 & -1 & 1 & 0 \\
0 & 1 & -1 & 0 \\
0 & 0 & 0 & 0 
\end{bmatrix}
& \mbox{and} &
M_1(\au)= \displaystyle{\frac{1}{a}}
\begin{bmatrix} 
0 & 0 & 0 & 0 \\
0 & 1 & 1 & 0 \\
0 & 1 & 1 & 0 \\
0 & 0 & 0 & 0 
\end{bmatrix}.
\end{eqnarray*}
Consider the operator $-\Delta(L,\mathds{1})$. The eigenvalues of $L$ are $0$ and $-\tfrac{2}{\gamma}$. The eigenvalues of $L(0,\au)= L+M(0,\au)$ are 
\begin{eqnarray*}
0, \frac{-a-\gamma+ \sqrt{a^2+\gamma^2}}{a\gamma}, -\frac{a+\gamma+ \sqrt{a^2+\gamma^2}}{a\gamma}  & \mbox{and} \ \displaystyle{-\frac{2}{\gamma}}. 
\end{eqnarray*}
Assume for simplicity from now on that $\gamma<0$. Then there are two positive eigenvalues
\begin{eqnarray*}
\displaystyle{-\frac{2}{\gamma}} & \mbox{and} &  \displaystyle{-\frac{a+\gamma+ \sqrt{a^2+\gamma^2}}{a\gamma}}
\end{eqnarray*}
of $L(0,\au)$ and therefore by Proposition~\ref{Luger} there are two negative eigenvalues $-\kappa_1^2$ and $-\kappa_2^2$ of $-\Delta(L,\mathds{1})$. The eigenvalues of $R(0,\au)= L+M_1(\au)$ are 
\begin{eqnarray*}
0, \frac{-a+\gamma+ \sqrt{a^2+\gamma^2}}{a\gamma}, -\frac{a-\gamma+ \sqrt{a^2+\gamma^2}}{a\gamma}  & \mbox{and} \ \displaystyle{-\frac{2}{\gamma}}. 
\end{eqnarray*}
For $\gamma<0$ there are two positive eigenvalues
\begin{eqnarray*}
\displaystyle{-\frac{2}{\gamma}} & \mbox{and} & \displaystyle{ -\frac{a-\gamma+ \sqrt{a^2+\gamma^2}}{a\gamma}} 
\end{eqnarray*}
of $R(0,\au)$. Hence by  Theorem~\ref{theorem3LB} 
\begin{eqnarray*}
\displaystyle{-\frac{a+\gamma+ \sqrt{a^2+\gamma^2}}{a\gamma} \leq \kappa_2 \leq -\frac{2}{\gamma} } & \mbox{and }  \displaystyle{ -\frac{2}{\gamma}  \leq \kappa_1 \leq -\frac{a-\gamma+ \sqrt{a^2+\gamma^2}}{a\gamma}}. 
\end{eqnarray*}
As already remarked for $a\to \infty$ the negative eigenvalues resemble the behaviour of the negative eigenvalues of an operator on a disjoint union of star graphs. Here this means that
\begin{eqnarray*}
\displaystyle{\lim_{a\to\infty}-\kappa_2^2=-\tfrac{4}{\gamma^2}} & \mbox{and } \displaystyle{ \lim_{a\to\infty}-\kappa_1^2=-\tfrac{4}{\gamma^2}}. 
\end{eqnarray*}
For small edge length the behaviour is more complicated. Since 
\begin{eqnarray*}
\displaystyle{\lim_{a\to 0}-\frac{a+\gamma+ \sqrt{a^2+\gamma^2}}{a\gamma}} = -\tfrac{2}{\gamma}  & \mbox{and} & \displaystyle{\lim_{a\to 0}-\frac{a-\gamma+ \sqrt{a^2+\gamma^2}}{a\gamma} = \infty }
\end{eqnarray*}
one has on the one hand that 
\begin{eqnarray*}
\displaystyle{\lim_{a\to 0}-\kappa_2^2=-\tfrac{4}{\gamma^2} } & \mbox{and} & \displaystyle{-\kappa_1^2 \leq -\tfrac{4}{\gamma^2} },
\end{eqnarray*}
but on the other hand it is not clear whether $-\kappa_1^2$ is bounded from below for $a\to 0$. A direct computation gives further information on $-\kappa_1^2$. Consider
\begin{eqnarray*}
\mathfrak{l}(\kappa)[x]= \langle L(\kappa,\au)x,x \rangle & \mbox{with} & x= \begin{bmatrix} 0 \\ 1 \\ 1 \\ 0 \end{bmatrix},
\end{eqnarray*}
which gives
\begin{eqnarray*}
\mathfrak{l}(\kappa)[x]= -\frac{1}{\gamma} - \kappa \tanh(a\kappa/2).
\end{eqnarray*}
The solution of $-\tfrac{1}{\gamma} - \kappa \tanh(a\kappa/2)=0$ is $\nu(-\gamma^{-1},a)$, which goes to infinity for $a\to 0$. From the variational characterization of the singular points of $L(\cdot,\au)$ in the forthcoming Theorem~\ref{VP} it follows that $\nu(-\gamma^{-1},a)\leq \kappa_1$ an therefore
$$\lim_{a\to 0}\kappa_1=\infty.$$
\end{example}

\begin{remark}
Since the self--adjoint operators $-\Delta(A,B)$ are self--adjoint and semi--bounded the operators $\Delta(A,B)$ generate quasi--contractive semigroups,
$$\norm{e^{t \Delta(A,B)}} \leq  e^{\omega t},$$
for appropriate $\omega$. If $-\Delta(A,B)$ has negative spectrum then the best possible choice is $\omega=\kappa_1^2$, that is the growth bound is exactly the absolute value of the smallest negative eigenvalue, compare for example \cite[Corollary II.3.6]{EngelNagel}. The above Theorems~\ref{theorem1LB}, \ref{theorem2LB} and \ref{theorem3LB} provide a priori estimates on this growth bound.   
\end{remark}

\section{Poincar\'{e} type inequalities on compact graphs}\label{secPoicare}
Assume now that $(\Ge,\au)$ is a compact metric graph, that is $\Ee=\emptyset$, and let $-\Delta(A,B)$ be a self-adjoint Laplace operator on this graph. It turns out that the eigenvalue zero of this operator is again related to the operator $L(\cdot,\au)$, which appeared in the study of the negative eigenvalues. 

\begin{proposition}\label{kernelLB}
Let $(\Ge,\au)$ be a compact metric graph and let $-\Delta(A,B)$ be a self--adjoint Laplace operator on the graph. Then zero is an eigenvalue of $-\Delta(A,B)$ if and only if zero is an eigenvalue of the operator $L(0,\au)$ considered as an operator in $\ran P^{\perp}$ and the equality 
$$\dim \Ker \left(-\Delta(A,B)\right)= \dim \Ker L(0,\au)$$
holds. In particular $-\Delta(A,B)$ is strictly positive if and only if $L(0,\au)$ as an operator in $\Ran P^{\perp}$ is strictly negative.  
\end{proposition}

As a consequence of Proposition~\ref{kernelLB} one can prove a criterion for having a Poincar\'{e} type inequality on certain subspaces of the Sobolev space $\mathcal{W}$.
\begin{theorem}\label{poincare}
Let $(\Ge,\au)$ be a compact metric graph and let $P$ be an orthogonal projector in $\Ke$. Whenever $\Ker \left(P^{\perp}M(0,\au)P^{\perp}+P\right)= \{0\}$ holds, there exists a constant $C>0$, where $C= C(P,\au)$, such that
$$\norm{u^{\prime}}_{\He} \geq C \norm{u}_{\He}$$
holds for all \index{$\We_P$}
$$u \in \We_P=\{\phi\in \mathcal{W} \mid P\underline{\phi}=0\}. $$
\end{theorem}

Consider for example a compact metric graph with at least one vertex of degree one. Impose at all vertices of degree larger than one the so--called Kirchhoff or standard boundary conditions, see for example \cite[Example 2.4]{VKRS2006} and Dirichlet boundary conditions on the vertices of degree one. Then the corresponding Laplacian is strictly positive and consequently a Poincar\'{e} type inequality holds.

\begin{proof}[Proof of Theorem~\ref{poincare}]
The operator $-\Delta(A,B)$ with $A=P$ and $B=P^{\perp}$ is self--adjoint, compare for example \cite{PKQG1} and since $L(0,\au)= P^{\perp}M(\kappa,\au)P^{\perp}\leq 0$ it follows from Theorem~\ref{Luger} that there are no negative eigenvalues. Since one has $\Ee=\emptyset$ the metric space $(\Ge,\au)$ is compact and from the Sobolev embedding theorem it follows that the spectrum of $-\Delta(A,B)$ is purely discrete. Hence under the assumption $\Ker \left(P^{\perp}M(0,\au)P^{\perp}+P\right)= \{0\}$ the operator $-\Delta(A,B)$ is even strictly positive by Proposition~\ref{kernelLB} and the infimum of the numerical range is attained by the lowest eigenvalue $$\lambda_1=\min \sigma (-\Delta(A,B))>0.$$ Consequently one has
\begin{eqnarray*}
\langle-\Delta(A,B)u,u\rangle\geq \lambda_1 \langle u,u\rangle, & \mbox{for all}\quad u\in \Dom(-\Delta(A,B)).
\end{eqnarray*}
The operator $-\Delta(A,B)$ is uniquely defined by the sesquilinear form $\bar{\delta}_{P}$ which is given by $\bar{\delta}_{P}[u,v]=\langle u^{\prime}, v^{\prime}\rangle$ on the form domain $\dom \bar{\delta}_{P}=\{\phi\in \mathcal{W} \mid P\underline{\phi}=0\},$ see \cite[Theorem 9]{PKQG1}. Since $\dom \bar{\delta}_{P}$ is the form closure of the operator domain of $-\Delta(A,B)$ the inequality 
$$\norm{u^{\prime}}^2=\langle u^{\prime},u^{\prime} \rangle\geq \lambda_1 \langle u,u\rangle $$ 
holds even for all $u\in \dom \bar{\delta}_{P}=\We_P$. The positive square root $k_1$ of the smallest positive eigenvalue $\lambda_1=k_1^2$ of $-\Delta(A,B)$ is a solution of the equation \\ $\det\left(PX(k,\au)+P^{\perp}Y(k,\au)\right)=0$, see for example \cite[Lemma 3.1]{VKRS2006}, and therefore it depends only on $P$ and $\au$. 
\end{proof}

\begin{proof}[Proof of Proposition~\ref{kernelLB}]
Eigenfunctions to the eigenvalue zero are piecewise affine. This gives the \textit{Ansatz}
\begin{eqnarray*}
\psi_0(x_j)= \alpha_j^0 + \beta_j^0x_j, & j\in \Ie 
\end{eqnarray*}
with traces 
\begin{align*}
\underline{\psi_0(\cdot)}= X_0\left(\underline{a} \right) \begin{bmatrix}  \{\alpha_j^0\}_{j\in \Ie} \\ \{\beta_j^0\}_{j\in \Ie} \end{bmatrix}, &&
\underline{\psi_0(\cdot)}^{\prime}= Y_0\left(\underline{a} \right) \begin{bmatrix}  \{\alpha_j^0\}_{j\in \Ie} \\ \{\beta_j^0\}_{j\in \Ie} \end{bmatrix},
\end{align*}
where
\begin{eqnarray*}
X_0\left(\underline{a}\right)=\begin{bmatrix} \mathds{1} & 0 \\ \mathds{1} & \au \end{bmatrix} & \mbox{and}& 
Y_0\left(\underline{a}\right)=\begin{bmatrix} 0 & \mathds{1} \\ 0 & -\mathds{1}  \end{bmatrix}.
\end{eqnarray*}
Consequently zero is an eigenvalue of the self--adjoint operator $-\Delta(A,B)$ if and only if 
\begin{align*}
\det \left(AX_0(\underline{a}) + BY_0(\underline{a})\right)=0.
\end{align*}
As $X_0\left(\underline{a} \right)$ is invertible this condition is equivalent to 
\begin{eqnarray*}
\det(A  + B M_0(\au))=0, & \mbox{where} & M_0(\au)=Y_0(\underline{a})X_0(\underline{a})^{-1}
\end{eqnarray*}
is exactly $M_0(\underline{a})=M(0,\underline{a})$, the operator from equation \eqref{M0LB}. Hence $-\Delta(A,B)$ has eigenvalue zero if and only if zero is an eigenvalue of $L(0,\au)$ considered as an operator in $\Ran P^{\perp}$ and the multiplicities of both agree. Combining this with Proposition~\ref{Luger} one obtains that $-\Delta(A,B)$ is strictly positive if and only if zero is no eigenvalue and there are no negative eigenvalues, this is the case if and only if $L(0,\au)$ is strictly negative.  
\end{proof}

\section{Variational methods and proofs}
An appropriate  method to deal with the negative eigenvalues of $-\Delta(A,B)$ is the variational principle developed by P.~A.~Binding, D.~Eschw{\'e} and H.~Langer in \cite{HL2000}, which has been extended by D.~Eschw{\'e} and M.~Langer in \cite{ML2004}. It has been successfully applied in the article \cite{AL2009} to compute the number of negative eigenvalues. Some facts are going to be revisited here as far as necessary. 

Let $I:=[\alpha,\beta)\subset \R$ be a real (not necessarily bounded) interval and and $X$ a finite dimensional Hilbert space. Denote by $\Herm\left(X\right) $ the set of all Hermitian operators in $X$. The spectrum of a function 
\begin{eqnarray*}
T(\cdot) \colon I \rightarrow \Herm\left(X\right),&  \lambda \mapsto T(\lambda) 
\end{eqnarray*}
is defined by 
$$\sigma(T(\cdot)):= \left\{ \lambda\in \C \mid \det T(\lambda)=0 \right\}.$$

The variational principle for operator valued functions is inspired by the $\min$--$\max$ principle for the linear eigenvalue problem. It has been exhibited in \cite{HL2000} that spectral points of more general operator valued functions can be found similarly to eigenvalues of linear self--adjoint operators as long as some key--properties are assumed. The results obtained in \cite{ML2004} are reduced and reformulated for the purpose of this work.

\begin{theorem}[{compare \cite[Theorem 2.1]{ML2004}}]\label{VP}
Let $T\colon [\alpha,\beta) \to \Herm\left(X\right), \lambda \mapsto T(\lambda)$ be 
\begin{enumerate}
\item norm--continuous and
\item  assume that for each $x\in X\setminus\{0\}$, the function 
$$\mathfrak{t}[x](\lambda):= \langle T(\lambda)x,x\rangle$$ is decreasing at value zero, which means that from $\mathfrak{t}[x](\lambda_0)=0$ it follows that for $\lambda<\lambda_0$, $\mathfrak{t}[x](\lambda)>0$ and for $\lambda>\lambda_0$, $\mathfrak{t}[x](\lambda)<0$ holds. 
\end{enumerate} 
Denote by $N_-$ the number of negative eigenvalues of $T(\alpha)$, by $N_+$ the number of positive eigenvalues of $T(\alpha)$ and by $N_0$ the dimension of $\Ker T(\alpha)$. Assume furthermore that 
\begin{itemize}
\item[(3)] there is a $\gamma\in[\alpha,\beta)$ such that $T(\gamma)<0$. 
\end{itemize}
Then $\sigma(T(\cdot))$ consists of $N_+$ eigenvalues
$\lambda_1 \leq \ldots \leq \lambda_n \leq \ldots \leq \lambda_{N_+}$ (counted with multiplicities), $\lambda_n>\alpha$ and $\alpha$ is an eigenvalue with multiplicity $N_0$. The eigenvalues $\lambda_n$ are given by
\begin{eqnarray*}
\displaystyle{\lambda_n= \min_{\dim S = N_-+N_0 +n} \max_{x\in S} \rho(x)} &\mbox{and} &
\displaystyle{\lambda_n= \max_{\dim S = n+N_-+N_0+1} \min_{x\perp S}  \rho(x), }
\end{eqnarray*}
where the generalized Rayleigh functional $\rho(x)$ is the unique solution of 
\begin{eqnarray*}
\mathfrak{t}[x](\cdot)=0, & \mbox{for } \norm{x}=1 
\end{eqnarray*}
and if there is no solution one sets $\rho(x)=-\infty$.
\end{theorem}  
The proof of the above theorem is based on the study of  the function 
$$\kappa\colon [\alpha, \beta) \rightarrow \N, \ \lambda\mapsto \kappa(\lambda):= \dim T_-(\lambda),$$ 
where $T_-(\lambda)$ denotes the strictly negative part of $T(\lambda)$. This function is monotone decreasing and left continuous and the height of jumps of $\kappa$ gives the multiplicity of an eigenvalue of $T(\cdot)$. An important corollary for the construction of appropriate comparison operators is the following theorem, which allows to compare two operator valued functions to each other whenever an inequality in terms of quadratic forms holds.

\begin{theorem}[{compare \cite[Corollary 2.12]{ML2004}}]\label{CVP}
Let $T(\cdot)$ and $S(\cdot)$ be two operator valued functions defined on $[\alpha,\beta)$ that satisfy the assumptions of Theorem~\ref{VP}. Denote by $\lambda^T_1\leq \ldots\leq\lambda^T_{N_+}$ and $\lambda^S_1\leq \ldots\leq\lambda^S_{M_+}$ the eigenvalues (counted with multiplicities) of $T(\cdot)$ and $S(\cdot)$ in $(\alpha,\beta)$, respectively. Assume that
$$\mathfrak{t}[x](\lambda)\leq \mathfrak{s}[x](\lambda)$$ 
holds for all $x\in X$ and all $\lambda\in [\alpha,\beta)$. Then one has $N_+\leq M_+$ and for the $N_+$ largest eigenvalues of $S(\cdot)$ and $T(\cdot)$ (counted with multiplicities) it follows that
\begin{eqnarray*}
\lambda^T_n \leq \lambda^S_n.
\end{eqnarray*}
\end{theorem}

\begin{lemma}\label{lemma1LB}
The operator valued function 
\begin{eqnarray*}
L(\cdot,\au) \colon [0,\infty) \rightarrow  \Herm(\ran P^{\perp}),
\end{eqnarray*}
defined by $L(\kappa,\au)$, satisfies the assumptions of Theorem~\ref{VP}. 
\end{lemma}
Multiplying $L(\kappa,\au)$ from both sides with the symmetry $Q$ one obtains the unitarily equivalent operator
\begin{align*}
L_Q(\kappa,\au)=QLQ +  QP^{\perp}Q D(\kappa,\au)QP^{\perp}Q,
\end{align*}\index{$L_Q$, $L_Q(\kappa,\au)$}\index{$P_Q$, $P_Q^{\perp}$} which is considered as an operator in the space $\Ran QP^{\perp}Q$. The operator $P_Q:=QPQ$ is an orthogonal projector with orthogonal complement $P_Q^{\perp}:=\mathds{1}-P_Q$, $P_Q^{\perp}=QP^{\perp}Q$ and $L_Q:=QLQ$ defines an operator in the space $\Ran P_Q^{\perp}$, which is isometrically isomorphic to $\Ran P^{\perp}$. 
\begin{proof}
The function $L(\cdot,\au)$ is norm--continuous, because the function $D(\cdot,\au)$ is already norm--continuous. It remains to prove that the function $\mathfrak{l}[x](\cdot)=\langle L(\cdot,\au)x,x \rangle$ is decreasing at the point zero for all $x\in \ran P^{\perp}$. This is implied by the statement that the function $\mathfrak{l}_Q[x](\cdot)=\langle L_Q(\cdot,\au)x,x \rangle$ is strictly decreasing for all $x\in \ran P_Q^{\perp}$ with $\norm{x}=1$. Since $D(\cdot,\au)$ is a diagonal matrix with strictly decreasing functions on the diagonal, the function defined by $\langle (P_Q+L_Q + D(\cdot,\au)) x,x \rangle$ is strictly decreasing for any $x\in \Ke$, in particular for all $x\in\ran P_Q^{\perp}$. Furthermore for any $x\in \Ran P_Q^{\perp}\setminus\{0\}$ one has $\langle L(k,\au)x,x \rangle \to -\infty$ for $k\to \infty$ and therefore also Assumption $(3)$ of Theorem~\ref{VP} is fulfilled.
\end{proof}

One considers different operator valued functions in order to compare them to $L(\cdot,\au)$ -- or equivalently to $L_Q(\cdot,\au)$ -- by means of Theorem~\ref{CVP}. For this purpose take into account

\begin{lemma}\label{UG}
Let be $0< a_{\min}\leq a_i \leq a_{\max}$. Then the following elementary estimates hold for $\kappa \geq 0$,
\begin{align}
-\kappa \tanh\left(\tfrac{\kappa a_{\max}}{2}\right) \ \ \ &\leq& -\kappa \tanh\left(\tfrac{\kappa a_i}{2}\right) \ \ \ &\leq& -\kappa \tanh\left(\tfrac{\kappa a_{\min}}{2}\right), \label{UG1}\\
-\frac{\kappa}{\tanh\left(\frac{\kappa a_{\min}}{2}\right)} \ \ \ &\leq& -\frac{\kappa}{\tanh\left(\frac{\kappa a_{i}}{2}\right)} \ \ \  &\leq& -\frac{\kappa}{\tanh\left(\frac{\kappa a_{\max}}{2}\right)}, \label{UG2a} \\
-\frac{\kappa}{\tanh\left(\frac{\kappa a_{\max}}{2}\right)}+\frac{2}{a_{\max}}  \ \ \ &\leq& -\frac{\kappa}{\tanh\left(\frac{\kappa a_{i}}{2}\right)}+\frac{2}{a_i} \ \ \  &\leq& -\frac{\kappa}{\tanh\left(\frac{\kappa a_{\min}}{2}\right)}+\frac{2}{a_{\min}}, \label{UG2} \\
-\frac{\kappa}{\tanh\left(\frac{\kappa a_{i}}{2}\right)} \ \ \  &<& -\kappa \  \ \ \  &\leq& -\kappa \tanh\left(\tfrac{\kappa a_{i}}{2}\right),  \label{UG3}\\
-\kappa \ \ \ &\leq& -\kappa \tanh\left(\tfrac{\kappa a_{i}}{2}\right) \ \ \ &\leq&   -\frac{\kappa}{\tanh\left(\frac{\kappa a_{i}}{2}\right)} +\frac{2}{a_i}.\label{UG4}
\end{align}
In \eqref{UG3} and \eqref{UG4} equality holds only for $\kappa=0$. In \eqref{UG1} equality holds only if $\kappa=0$ or if $a_i=a_{\min}=a_{\max}$. In \eqref{UG2a} and \eqref{UG2} equality holds only if $a_i=a_{\min}=a_{\max}$. 
\end{lemma}
\begin{proof}
The function $-\tanh(y)$ is strictly decreasing for $y\geq 0$. Therefore plugging in
\begin{eqnarray}\label{plugin}
y_1:=\frac{\kappa a_{\min}}{2} \leq \frac{\kappa a_{i}}{2}=:y_2
\end{eqnarray}
and multiplying with $\kappa\geq 0$ yields inequality \eqref{UG1}. With a similar calculation one obtains \eqref{UG2a}. Note that the function $\tfrac{-y}{\tanh(y)}+1$ is strictly decreasing for $y\geq 0$ and hence one obtains with \eqref{plugin} and by multiplying with 
\begin{eqnarray*}
\frac{2}{a_{\min}} \geq \frac{2}{a_{i}}
\end{eqnarray*}
the inequality \eqref{UG2}. Inequality \eqref{UG3} follows already from the inequality $\tanh(y)< 1$. The last inequality \eqref{UG4} is equivalent to $-y\tanh^2(y) \leq -y + \tanh(y)$ for $y\geq 0$, which in turn is true, because one has for the derivatives
\begin{eqnarray*}
\frac{d}{dy}\left(-y\tanh^2(y)\right) = -\tanh^2(y) - 2y \tanh(y)(1-\tanh(y)^2) 
\end{eqnarray*}
and 
\begin{eqnarray*}
\frac{d}{dy}\left(-y + \tanh(y)\right) = -\tanh^2(y),
\end{eqnarray*}
and for the initial values at $y=0$ the equality $-0\tanh^2(0)=0 = -0 + \tanh(0)$.  
\end{proof}

\begin{proof}[Proof of Theorem~\ref{theorem1LB}]
Taking into account
\begin{align*}
L_Q(\kappa,\au) =& L_Q + P^{\perp}_Q D(\kappa,\au) P^{\perp}_Q \\
=& L_Q + P^{\perp}_Q D(0,\au) P^{\perp}_Q + P^{\perp}_Q \left(D(\kappa,\au)-D(0,\au) \right) P^{\perp}_Q\\
=& L_Q(0,\au) + P^{\perp}_Q \left(D(\kappa,\au)-D(0,\au) \right) P^{\perp}_Q
 \end{align*}
one considers the operator valued functions
\begin{eqnarray*}
L_1(\kappa,\au) = L_Q(0,\au) - \left(\frac{\kappa}{\tanh(\kappa a_{\min}/2)} - \frac{2}{a_{\min}}\right)P_Q^{\perp} 
\end{eqnarray*}
and 
\begin{eqnarray*}
L_2(\kappa,\au) = L_Q(0,\au) - \kappa P_Q^{\perp}
\end{eqnarray*}
as comparison operators. One has by \eqref{UG2} and \eqref{UG4} that 
\begin{eqnarray*}
-\kappa \mathds{1} \leq D(k,\au)-D(0,\au) \leq - \left(\frac{\kappa}{\tanh(\kappa a_{\min}/2)} - \frac{2}{a_{\min}}\right) \mathds{1} & \mbox{for } \kappa\geq 0,
\end{eqnarray*}
and hence
\begin{eqnarray*}
L_2(\kappa) \leq L_Q(\kappa) \leq  L_1(\kappa) & \mbox{for } \kappa\geq 0.
\end{eqnarray*}
By definition the first $n$ positive eigenvalues of $L_2(\kappa)$ are $l_i$, for $i=1,\ldots,n$ and the first $n$ positive eigenvalues of $L_1(\kappa)$ are $\eta(l_i,a_{\min})$, for $i=1,\ldots,n$. The operator valued functions $L_1(\cdot)$ and $L_2(\cdot)$ are strictly decreasing and continuous. The proof of this is analogue to the one of Lemma~\ref{lemma1LB}. Theorem~\ref{CVP} delivers the estimates for the numbers $\kappa_i$.

To prove the optimality of the resulting lower bound on the spectrum one considers the spaces
\begin{eqnarray*}
E_{a_{\min}}:=\mbox{span} \left\{ \begin{bmatrix} 0 \\ 0 \\ e_i \end{bmatrix} \Bigg\vert a_i=a_{\min} \right\} &\mbox{and} & \Ker(L_Q(0,\au)-l_1),
\end{eqnarray*}
which is the eigenspace of $L_Q(0,\au)$ to its largest positive eigenvalue $l_1$. One proves that the bound $\eta(l_i,a_{\min})$ is optimal if and only if there is a vector $x\neq 0$ with
$$x \in \Ker(QLQ-l_1) \cap E_{a_{\min}}.$$ 
Assume that $x\in \Ker(L_Q(0,\au)-l_1)\cap E_{a_{\min}}$ with $\norm{x}=1$. Then one has 
$$\mathfrak{l}_Q[x](\kappa)=\langle x, L_Q(\kappa,\au)x\rangle = l_1 + \frac{2}{a_{\min}} - \frac{\kappa}{\tanh(\kappa a_{\min}/2)}.$$ 
Denote the unique zero of this function by $\kappa_0$, and observe that $\kappa_0=\eta(l_1,a_{\min})$. Assume that $y\neq 0$ and denote furthermore by $\rho(y)$ the solution of $\mathfrak{l}_Q[y](\cdot)=0$ or if there is no solution of this one sets $\rho(y)=-\infty$. Note that $\rho(y)\leq \rho(x)$ holds for all $y \in \ran P_Q^{\perp}$ with $\norm{y}=1$, because on the one hand $l_1\geq \langle y, L_Q(0,\au) y\rangle$ holds and on the other hand for the inequalities \eqref{UG1}, \eqref{UG2} and \eqref{UG3} 
$$\frac{2}{a_{\min}} - \frac{\kappa}{\tanh(\kappa a_{\min}/2)} \geq \langle y, \left(D(\kappa,\au)-D(0,\au)\right) y\rangle$$ holds. Since there is equality for $y=x$ one concludes 
\begin{align*}
\rho(x)= \max_{\underset{\norm{y}=1}{y\in \ran P_Q^{\perp}}} \rho(y).
\end{align*}
By the variational characterization of the eigenvalues given in Theorem~\ref{VP} it follows that $\kappa_0$ is indeed a zero of \\ $\det L_Q(\cdot,\au)$ and hence $-\kappa_0^2$ is the lowest eigenvalue of $-\Delta(A,B)$.

Conversely assume that the bound $\eta(l_1,a_{\min})$ is taken, which means that there exists a vector $x\in \ran P^{\perp}$ with $\norm{x}=1$ such that $L_Q(\eta(l_1,a_{\min}),\au)x=0$.  Consider again the function 
$$\mathfrak{l}_Q[x](\kappa)= \langle L_Q(\kappa,\au)x,x\rangle .$$ 
One first shows that $$\mathfrak{l}_Q[x](0) = l_1 $$ holds. Assume that $\mathfrak{l}_Q[x](0) < l_1$. Since $D(\kappa,\au)-D(0,\au) \leq \frac{2}{a_{\min}}- \frac{\kappa}{\tanh(\kappa a_{\min}/2)}$ for $\kappa\geq 0$ it would follow that the unique solution of $\mathfrak{l}_Q[x](\kappa)=0$ is smaller than $\eta(l_1,a_{\min})$, which is a contradiction to the assumption. 
\\
Assume conversely that $\mathfrak{l}_Q[x](0) > l_1$. This is a contradiction to the inequality $L_Q(\kappa,\au) \leq  L_1(\kappa)$ for $\kappa\geq 0$ and the claim follows.

Since 
$$l_1= \max_{\underset{\norm{y}=1}{y\in \ran P_Q^{\perp}}}  \langle y, L(0,\au)y \rangle $$ it follows from $\mathfrak{l}_Q[x](0) = l_1$ by the classical $\min-\max$--principle that \\ $x\in \Ker(L_Q(0,\au)-l_1)$.

Assume now that $x\notin E_{a_{\min}}$. One has by \eqref{UG2} and \eqref{UG4} that 
\begin{eqnarray*}
\langle (D(\kappa,\au)-D(0,\au))x,x\rangle \leq \frac{2}{a_{\min}}- \frac{\kappa}{\tanh(\kappa a_{\min}/2)}, & \mbox{for } \kappa> 0 \ \mbox{with } \norm{x}=1,
\end{eqnarray*}
and that equality holds only if $x\in E_{a_{\min}}$. Hence $\mathfrak{l}_Q[x](\kappa) < l_1 - \frac{2}{a_{\min}} - \frac{\kappa}{\tanh(\kappa a_{\min}/2)}$ would hold for $\kappa >0$ and it would follow that the unique solution of $\mathfrak{l}_Q[x](\kappa)=0$ was smaller than $\eta(l_1,a_{\min})$. This is a contradiction and hence $x\in E_{a_{\min}}$. Note that 
$$QE_{a_{\min}}:=\mbox{span} \left\{ \begin{bmatrix} 0 \\ e_i \\ -e_i \end{bmatrix} \Bigg\vert a_i=a_{\min} \right\}.$$
\end{proof}

\begin{proof}[Proof of Theorem~\ref{theorem2LB}]
Define the comparison operator
\begin{eqnarray*}
L_3(\kappa,a_{\min}) = L_Q - \kappa \tanh(\kappa a_{\min}/2)P_Q^{\perp} 
\end{eqnarray*}
By definition the first $n$ positive eigenvalues of $L_3(\kappa,a_{\min})$ are $\nu(m_i,a_{\min})$, for $i=1,\ldots,n$. It is straight forward to verify that the operator valued function $L_3(\cdot,a_{\min})$ is strictly decreasing and continuous. From the inequalities \eqref{UG1} and \eqref{UG3} one reads
\begin{eqnarray*}
D(k,\au) \leq - \kappa \tanh(\kappa a_{\min}/2) \mathds{1} & \mbox{for } \kappa\geq 0,
\end{eqnarray*}
and consequently
\begin{eqnarray*}
L_Q(\kappa,\au) \leq  L_3(\kappa,a_{\min}), & \mbox{for } \kappa\geq 0.
\end{eqnarray*}
Applying Theorem~\ref{CVP} proves the first part of the theorem. 

To prove the optimality of the resulting lower bound on the spectrum one considers the spaces
\begin{eqnarray*}
F_{a_{\min}}:=\mbox{span} \left\{ \begin{bmatrix} 0 \\ e_i \\ 0 \end{bmatrix} \Bigg\vert a_i=a_{\min} \right\} &\mbox{and} & \Ker(L_Q-m_1),
\end{eqnarray*}
which is the eigenspace of $L_Q$ to its largest positive eigenvalue $m_1$. One proves as in the proof of Theorem~\ref{theorem1LB} that the bound $-\nu(m_i,a_{\min})^2$ is optimal if and only if there is a vector $x\neq 0$ with
$$x \in \Ker(L_Q-m_1) \cap F_{a_{\min}}.$$ 

To prove the lower bounds on the numbers $\kappa_i$ consider
\begin{eqnarray*}
L_4(\kappa,a_{\min}) = L_Q - \left( \frac{\kappa}{\tanh(\kappa a_{\min}/2)}\right)P^{\perp}_Q. 
\end{eqnarray*}
For each $m_i$ with $m_i > \tfrac{2}{a_{\min}}$ there is one eigenvalue of $L_4(\cdot,a_{\min})$, which is by definition $\eta(m_i-\tfrac{2}{a_{\min}},a_{\min})$. Since by \eqref{UG2a} and \eqref{UG3}
\begin{eqnarray*}
\frac{-\kappa}{\tanh(\kappa a_{\min}/2)}\mathds{1} \leq D(k,\au) & \mbox{for } \kappa\geq 0,
\end{eqnarray*}
also
\begin{eqnarray*}
L_4(\kappa,a_{\min})\leq L_Q(\kappa,\au), & \mbox{holds for } \kappa\geq 0,
\end{eqnarray*}
and the claim follows with Theorem~\ref{CVP}. 
\end{proof}

\begin{proof}[Proof of Theorem~\ref{theorem3LB}]
The proof of Theorem~\ref{theorem3LB} takes advantage of the estimates 
\begin{eqnarray*}
-\kappa + \frac{2}{a}\geq  -\kappa\tanh\left(\frac{a}{2}\kappa\right) & \mbox{for } \kappa \geq 0
\end{eqnarray*}
which is equivalent to $y-1 \leq y\tanh(y)$ for $y\geq 0$. Together with \eqref{UG3} this yields for $\kappa \geq 0$ the inequality
\begin{eqnarray*}\index{$M_1(\kappa,\au)$}
D(\kappa,\au) \leq  M_1(\kappa,\au), & \mbox{where }   M_1(\kappa,\au) =  \begin{bmatrix} -\kappa & 0 & 0 \\ 0 & -\kappa + \tfrac{2}{\au} & 0 \\ 0 & 0 & -\kappa   \end{bmatrix}.
\end{eqnarray*}
This in turn gives together with the lower estimates from Theorem~\ref{theorem1LB} the inequality 
$$ L(0,\au)-\kappa P^{\perp} \leq  L(\kappa,\au) \leq R(\kappa,\au).$$ 
Applying Theorem~\ref{CVP} yields the claim. 
\end{proof}

\end{document}